\documentclass[english,a4paper]{amsart}

\usepackage[backref=page]{hyperref}

\renewcommand*{\backref}[1]{}
\renewcommand*{\backrefalt}[4]{%
	\ifcase #1 Not cited.%
	\or        Cited on page~#2.%
	\else      Cited on pages~#2.%
	\fi}

\usepackage[a4paper,lmargin=2.0cm,rmargin=2.0cm,tmargin=2.0cm,bmargin=2.0cm]{geometry}

\usepackage{parskip}

\usepackage[utf8]{inputenc}
\usepackage{amssymb}
\usepackage{amsmath}
\usepackage{amsthm}
\usepackage{rotating}
\usepackage{color}
\usepackage{babel}

\usepackage{array}
\usepackage{url}
\usepackage[ruled]{algorithm2e}

\newcolumntype{H}{>{\setbox0=\hbox\bgroup}c<{\egroup}@{}}

\SetKwProg{Fn}{Function}{ is}{end}
\SetKwFunction{Print}{Print}

\newcommand{\repourl}{\url{https://github.com/amato-gianluca/weirds}}
\newcommand{\eqdef}{:=}
\newcommand{\sigmam}{\sigma_{-1}}

\newtheorem{thm}{Theorem}[section]
\newtheorem*{thm-parma}{Theorem 3.1 from \cite{AHMPWN}}
\newtheorem*{thm*}{Theorem}
\newtheorem{lem}[thm]{Lemma}
\newtheorem{cor}[thm]{Corollary}
\newtheorem{conj}[thm]{Conjecture}
\newtheorem{prop}[thm]{Proposition}

\theoremstyle{definition}
\newtheorem{defin}[thm]{Definition}
\newtheorem{example}[thm]{Example}

\theoremstyle{remark}
\newtheorem{remark}[thm]{Remark}
\newtheorem{openq}[thm]{Open Question}

\newcommand{\NN}{\mathbb{N}}
\renewcommand{\div}{\mid}
\newcommand{\entdiv}{\mid\mid}

\newcommand{\ap}[1]{P_{#1}}

\definecolor{darkgreen}{rgb}{0,0.5,0}

\title{Primitive abundant and weird numbers with many prime factors}

\author[G. Amato]{Gianluca Amato}
\address{Universit\`a di Chieti-Pescara\\ Dipartimento di Economia Aziendale, viale della Pineta 4, I-65129 Pescara, Italy}
\email{gianluca.amato@unich.it}
\author[M. F. Hasler]{Maximilian F. Hasler}
\address{Universit\'e des Antilles\\D.S.I., B.P. 7209\\ Campus de Schoelcher\\ F-97275 Schoelcher cedex, Martinique (F.W.I.)}
\email{Maximilian.Hasler@Univ-Antilles.fr}
\author[G. Melfi]{Giuseppe Melfi}
\address{University of Applied Sciences of Western Switzerland\\ HEG Arc\\
Espace de l'Europe, 21, CH-2000 Neuch\^atel}
\email{giuseppe.melfi@he-arc.ch}
\author[M. Parton]{Maurizio Parton}
\address{Universit\`a di Chieti-Pescara\\ Dipartimento di Economia, viale della Pineta 4, I-65129 Pescara, Italy}
\email{maurizio.parton@unich.it}

\thanks{The third and the fourth author were supported by the GNSAGA group of INdAM and by the MIUR under the PRIN Project ``Variet\`a reali e complesse: geometria, topologia e analisi armonica''}



\begin{document}

\begin{abstract}
We give an algorithm to enumerate all primitive abundant numbers (briefly, PAN) with a fixed $\Omega$ (the number of prime factors counted with their multiplicity), and explicitly find all PAN up to $\Omega=6$, count all PAN and square-free PAN up to $\Omega=7$ and count all odd PAN and odd square-free PAN up to $\Omega=8$. We find primitive weird numbers (briefly, PWN) with up to 16 prime factors, improving the previous results of \cite{AHMPWN} where PWN with up to 6 prime factors were given. The largest PWN we find has 14712 digits: as far as we know, this is the largest example existing, the previous one being 5328 digits long \cite{MelCIP}. We find hundreds of PWN with exactly one square odd prime factor: as far as we know, only five were known before. We find all PWN with at least one odd prime factor with multiplicity greater than one and $\Omega  = 7$ and prove that there are none with $\Omega < 7$. Regarding PWN with a cubic (or higher power) odd prime factor, we prove that there are none with $\Omega\le 7$, and we did not find any with larger $\Omega$. Finally, we find several PWN with 2 square odd prime factors, and one with 3 square odd prime factors. These are the first such examples.
\end{abstract}

\maketitle
\section{Introduction}

Let $n\in\NN$ be a natural number, and let $\sigma(n)=\sum_{d\div n}d$ be the sum of its divisors. If $\sigma(n)>2n$, then $n$ is called \emph{abundant}, whereas if $\sigma(n)<2n$, then $n$ is called \emph{deficient}. \emph{Perfect numbers} are those for which $\sigma(n)=2n$.  If $n$ is abundant and can be expressed as a sum of distinct proper divisors, then $n$ is called \emph{semiperfect}, or sometimes also \emph{pseudoperfect}. A \emph{weird number} is a number which is abundant but not semiperfect.

If $n$ is abundant and it is not a multiple of a smaller non-deficient number than $n$ is called a \emph{primitive abundant number}, PAN in this paper. Similarly, a \emph{primitive weird number}, PWN in this paper, is a weird number which is not a multiple of any smaller weird number.

In this paper, we study primitive abundant and weird numbers.


Leonard Eugene Dickson, in two papers from 1913 on the American Journal of Mathematics~\cite{DicFOP,DicEAN}, proves that the sets of PAN having any given number $\omega$ of distinct prime factors is finite (for even PAN, one also needs to fix the exponent of 2). He then explicitly finds all odd PAN with $\omega\le 4$, and all even PAN with $\omega\le 3$, see also~\cite{FerMTE,HerTaE} for errata in Dickson's tables. The number of odd PAN with $\omega=5$ was found in 2017, while for $\omega\ge 6$ the problem is still open (see~\cite{DicAbN}, where a lower bound for $\omega=6$ is given). 

Motivated by the above discussion, in this paper we focus on the set of PAN with a given number $\Omega$ of prime factors \emph{counted with their multiplicity}. In this way we have been able to explicitly find all PAN with $\Omega$ up to 6, to count PAN with $\Omega\le 7$, and to count odd PAN with $\Omega\le 8$, see Table~\ref{tab:pan-count} and OEIS sequences \href{https://oeis.org/A298157}{A298157} and \href{https://oeis.org/A287728}{A287728}.
All these results are new, to the best of our knowledge.

Weird numbers were defined in 1972 by Stan Benkoski \cite{BenPSS}, and appear to be rare: for instance, up to $10^4$ we have only $7$ of them \cite{SloOEI}. Despite this apparent rarity, which is the reason for the name, weird numbers are easily proven to be infinite: if $n$ is weird and $p$ is a prime larger than $\sigma(n)$, then $np$ is weird (see for example \cite[page~332]{FriSDE}). But a much stronger property is true: Benkoski and Paul Erd\H{o}s, in their joint 1974 paper \cite{BeEWPN}, proved that the set of weird numbers has positive asymptotic density. Nevertheless, it is not yet known whether PWN are infinite:
\begin{conj}\cite[end of page 621]{BeEWPN}
	There exist infinitely many PWN.
\end{conj}

The search for new PWN is made more interesting by the fact that we still do not know whether they are infinite. The list of the first PWN is regularly updated by \href{https://oeis.org/wiki/User:Robert_G._Wilson_v}{Robert G. Wilson v}, and at the time of writing (February 2018) the first 1081 PWN are known, see OEIS \href{https://oeis.org/A002975}{A002975}.


Looking for the largest possible PWN is also very interesting. One approach is to consider patterns in the prime factorization of PWN, see~\cite{AHMPWN}.
At the time of writing (February 2018), only 12 PWN with 6 distinct prime factors~\cite{AHMPWN} and just one PWN with 7 distinct prime factors~\cite{BREPCL} are known.

In this paper we dramatically improve these figures. We find hundreds of PWN with more than 6 distinct prime factors. In particular, we find PWN with up to 16 distinct prime factors, see Tables~\ref{table:smallweirds} and \ref{table:bigweirds}. The largest PWN we have found has 16 distinct prime factors and 14712 digits. As far as we know, this is the largest PWN known, the previous one being 5328 digits long \cite{MelCIP}.

Another strange behavior in the prime decomposition of PWN is the fact that only five PWN with non square-free odd prime factors were known, see OEIS sequence \href{https://oeis.org/A273815}{A273815}, and no PWN with an odd prime factor of multiplicity strictly greater than two is known.
We explain this fact with Theorem~\ref{thm:patterns}:
\begin{thm*}[PWN with non square-free odd part and $\Omega\le 7$]
	\label{thm:patterns}
	There are no PWN $m$ with a quadratic or higher power odd prime factor and $\Omega(m)<7$. There are no PWN $m$ with 2 quadratic odd prime factor and $\Omega(m)=7$. There are no PWN $m$ with a cubic or higher power odd prime factor and $\Omega(m)=7$.
\end{thm*}
These results are new, to the best of our knowledge.

Finally, in this paper we find hundreds of new PWN with a square odd prime factor, see Table~\ref{table:primesquare} for a selection of them. We find several new PWN with 2 square odd prime factors, and one with 3 square odd prime factors, see Table~\ref{table:primesquaremore}. These are the first examples of this kind.

In the following, we describe with some details the methods of the paper.

In Section~\ref{sec:pre} we start with a careful analysis of the set $A_\Omega$ of PAN $m$ with a fixed number $\Omega=\Omega(m)$ of prime factors counted with their multiplicity. These sets are finite, a corollary of Dickson's theorems~\cite{DicFOP,DicEAN}, see also Theorem~\ref{thm:dick}. The main result in this section is Theorem~\ref{thm:primitive-weird} about the structure of PAN: they are of the form $mp^e$, where $p$ is a prime larger than the largest prime factor of $m$, and $m$ is a deficient number satisfying certain conditions involving the \emph{center} $c(m)=\sigma(m)/(2m-\sigma(m))$ of $m$, see Definition~\ref{def:center}. Note that some results in this section are easy consequences of the definitions, and some of them are well-known: however, we leave them in the paper because we will use them extensively in Sections~\ref{sec:abundant-enumerate} and~\ref{sec:weird}.


In Section~\ref{sec:abundant-enumerate} we face the problem of explicitly computing $A_\Omega$, or some statistics on it, for specific values of $\Omega$. Here we distinguish the square-free case from the general case, since the former appears to be notably simpler than the latter.

Every square-free PAN is then given by $p_1\dotsm p_{k-1}p_k$ for certain primes $p_1<\dots<p_k$, and $p_1\dotsm p_i$ is recursively built from $p_1\dotsm p_{i-1}$ by imposing $p_1\dotsm p_i$ deficient. This gives an explicit construction for $A_\Omega$ in the square-free case, see Algorithm~\ref{algo:sfpan}. However, since the condition for $p_1\dotsm p_i$ to be deficient is open (see Proposition~\ref{prop:deficient}), we need a termination condition. This is done by exploiting Theorem~\ref{thm:primitive-expand}, stating that a deficient sequence of primes $m = \bar p_1^{e_1} \dotsm \bar p_r^{e_r}$ can be completed to a PAN $mpq$ for suitable primes $p,q$. Applying this machinery we explicitly find all square-free PAN with $\Omega\le 6$, count the square-free PAN with $\Omega\le 7$ and count odd square-free PAN with $\Omega\le 8$, see Table~\ref{tab:sfpan-count} and OEIS sequences \href{https://oeis.org/A295369}{A295369} and \href{https://oeis.org/A287590}{A287590}.
%
%

Adapting the techniques to the non square-free case essentially means allowing consecutive primes in $m = p_1, \ldots, p_k$ to be equal, and being more careful in identifying which sequences of primes give origin to PAN. As already said, we explicitly find all PAN with $\Omega\le 6$, count PAN with $\Omega\le 7$, and count odd PAN with $\Omega\le 8$, see Table~\ref{tab:pan-count} and OEIS sequences \href{https://oeis.org/A298157}{A298157} and \href{https://oeis.org/A287728}{A287728}.


In Section~\ref{sec:weird} we deal with PWN. Since we are interested in $k$ prime factors for large $k$, it is not computationally feasible to explicitly find all PAN and then check for weirdness. Thus, when computing the deficient seed $m=p_1 \dotsm p_{k-1}$, we choose an \emph{amplitude} $a$ and limit the choice for $p_i$ to the first $a$ primes larger than $c(p_1\dotsm p_{i-1})$, and the choice for $p_k$ to the last $a$ primes smaller than $c(p_1\dotsm p_{k-1})$. In order to be able to deal with the huge numbers involved, we represent them in a form we have called \emph{index sequence}, see Definition~\ref{def:index}. Finally, in Remarks~\ref{rem:limitation},~\ref{rem:mfour}, and in Section~\ref{sec:open}, we explain what we noticed from experiments, as a possibly useful future reference.
The new findings in this section are: PWN with more than 6 distinct prime factors, PWN with up to 16 distinct prime factors, see Tables~\ref{table:smallweirds} and \ref{table:bigweirds}, PWN with a square odd prime factor, see Table~\ref{table:primesquare}, PWN with 2 and 3 square odd prime factors, see Table~\ref{table:primesquaremore}, and Theorem~\ref{thm:patterns} on patterns for PWN. 

Note that the problem of finding a PWN with a cubic or higher power odd prime factor is still open. This, and other open questions, are listed in Section~\ref{sec:open}.

All the software we have developed and results of our experiments are available on-line at the GitHub repository \repourl.


\section{Deficient, perfect and abundant numbers}\label{sec:pre}

In line with~\cite{GuyUPN}, we will refer to $\Delta(n) \eqdef \sigma(n)-2n$ as the \emph{abundance} of $n$, and to $d(n) \eqdef 2n-\sigma(n)=-\Delta(n)$ as the \emph{deficiency} of $n$. It is sometimes convenient to use the notation $\sigma_\ell(n) \eqdef \sum_{d\mid n}d^\ell$ for the sum of the $\ell$-th powers of divisors, so that $\sigma_0$ is the number of divisors including 1 and the number itself, $\sigma \eqdef \sigma_1$ their sum, and $\sigma_{-1}(n) = \sigma(n)/n$ is the \emph{abundancy} of $n$. One can characterize deficient, perfect and abundant numbers respectively by $\sigma_{-1}(n)<2$, $\sigma_{-1}(n)=2$, $\sigma_{-1}(n)>2$.

If $n = p_1^{e_1} \dotsm p_k^{e_k}$ with $p_1 < \dotsb < p_k$ primes, then for each $p_i$ we can choose an exponent from $0$ to $e_i$ to build a divisor of $n$. Therefore, the function $\sigma_\ell$ is multiplicative on prime factors, that is, we have:
\begin{equation}
\label{eq:sigma}
\sigma_\ell(n) = \prod_{i=1}^k  (1 + p_i^\ell + \dots + p_i^{\ell e_i}) = \prod_{i=1}^k \frac{p_i^{(e_i+1)l}-1}{p_i^\ell - 1} = \prod_{i=1}^k \sigma_\ell(p_i^{e_i})
\end{equation}
Moreover, since $\sigma_\ell(p^e) \leq \sigma_\ell(p)^e$, then $\sigma_\ell$ is sub-multiplicative, i.e., $\sigma_\ell(mn) \leq \sigma_\ell(m) \sigma_\ell(n)$.

If a number is non-deficient (i.e., either perfect or abundant) and all of its proper divisors are deficient, then it is called \emph{primitive non-deficient}. A \emph{primitive abundant number} PAN is a primitive non-deficient number which is also abundant\footnote{Some authors define a PAN to be an abundant number with no abundant proper divisors. The two definitions differ on multiples of perfect numbers. For example, $30 = 2 \cdot 3 \cdot 5$ is primitive abundant according to this alternative definition, but not according to ours, since $2 \cdot 3$ is perfect, hence non-deficient.}.

\begin{prop}
	\label{prop:abundance}
	If $m$ is non-deficient and $n \in \NN$, $n>1$, then $mn$ is abundant.
\end{prop}
\begin{proof}
	If $d$ and $d'$ are distinct divisors of $m$, then $dn$ and $d'n$  are distinct divisors of $mn$. If $\sum_{d \div m} d \geq 2m$, then $\sum_{d \div mn} d > \sum_{d \div m} dn \geq 2mn$.
\end{proof}

\begin{cor}
	All perfect numbers are primitive non-deficient.
\end{cor}
\begin{proof}
	If $m$ is perfect and $n$ is a proper divisor of $m$, $n$ should be deficient. Otherwise, by Proposition~\ref{prop:abundance}, $m$ would be  abundant.
\end{proof}

The following corollary states that, whenever we want to check if $m$ is primitive, we need to look only at a subset of its divisors.
\begin{cor}
	\label{cor:primdiv}
	If $m$ is abundant and $m/p$ is deficient for all primes $p \div m$, then $m$ is primitive abundant.
\end{cor}
\begin{proof}
	If $m$ is not primitive abundant, then there exists some non-deficient number $m'\mid m$, with $m'<m$. Moreover, there exists a prime $p \div m$ such that $m' \div m/p$, and since $m/p$ is deficient, by Proposition~\ref{prop:abundance} this contradicts the fact that $m'$ is non-deficient.
\end{proof}

\begin{prop}
	\label{prop:ab1}
	Let $m = p_1^{e_1} \dotsm p_k^{e_k}$ with $p_1 <  \dots < p_k$. Choose a position $i \leq k$ and a prime $p$ such that $(m,p)=1$. Let $\widetilde m$ be the result of substituting $p_i^{e_i}$ with $p^{e_i}$ in the decomposition of $m$, i.e., $\widetilde m = mp^{e_i}/p_i^{e_i}$. Then
	\begin{itemize}
		\item\label{it:tildeabu} if $m$ is abundant or perfect and $p < p_i$, then $\widetilde m$ is abundant;
		\item\label{it:tildedef} if $m$ is deficient or perfect and $p > p_i$ then $\widetilde m$ is deficient.
	\end{itemize}
\end{prop}
\begin{proof}
	Note that replacing $p_i$ with $p$ in $m$ results in replacing $p_i$ with $p$ in all divisors of $m$. This means that whenever $p_i$ appears in a summand of $\sigma_{-1}(m)$, it is replaced with $p$. Thus, $\sigma_{-1}(m)$ is decreasing in the $p_i$'s.
	Therefore, if $m$ is abundant or perfect and $p < p_i$, we have $\sigma_{-1}(\widetilde m) > \sigma_{-1}(m) \geq 2$, hence $\tilde m$ is abundant. Similarly for the second case.
\end{proof}

Note that, if $m =  p_1^{e_1} \dotsm p_k^{e_k}$ is primitive abundant and we replace $p_i^{e_i}$ with $p^{e_i}$ for some $p < p_i$, we are not sure that the number we obtain is primitive abundant (although we know it is abundant). For example, $3^2 \cdot 5\cdot 7 \cdot 103$ is primitive abundant, but $2^2 \cdot 5 \cdot 7 \cdot 103$ is not, since $2 \cdot 5 \cdot 7$ is primitive abundant. Another example  involving square-free numbers is the following: $2 \cdot 7 \cdot 11 \cdot 13$ is primitive abundant but $2 \cdot 5 \cdot 11 \cdot 13$ is not, since $2 \cdot 5 \cdot 11$ is already abundant.

\subsection{Adding a new coprime factor \texorpdfstring{$p^e$}{p\textasciicircum e} to a deficient number}

The following proposition considers the problem of starting with a deficient number $m$ and adding a new prime factor $p^e$ with $(p,m)=1$. We want to study under which conditions $mp^e$ is perfect, (primitive) abundant or deficient. The reason we are interested in this problem is that, in Section~\ref{sec:abundant-enumerate}, we will build PAN by adjoining one prime factor at a time to a starting deficient number.

\begin{prop}
	\label{prop:deficient}
	If $m$ is deficient, $e \in \NN$ and $p$ is a prime such that $(m,p)=1$, 
	\begin{itemize}
		\item if $p^e / \sigma(p^{e-1}) < \sigma(m) / d(m)$ then $mp^e$ is abundant;
		\item if $p^e / \sigma(p^{e-1}) = \sigma(m) / d(m)$ then $mp^e$ is perfect;
		\item if $p^e / \sigma(p^{e-1}) > \sigma(m) / d(m)$ then $mp^e$ is deficient.
	\end{itemize}
\end{prop}
\begin{proof}
	With simple algebraic manipulations one can show that $d(mp^e) = d(m)p^e - \sigma(m)\sigma(p^{e-1})$. The thesis immediately follows by imposing $d(mp^e)$ less, equal or greater than zero.
\end{proof}

Since the term $\sigma(m)/d(m)$ will have a major role in the following, we introduce a more succinct notation. See also~\cite[Formula~(10)]{DicFOP}.

\begin{defin}[Center of a deficient number]
	\label{def:center}
	Given  $m \in \NN$, we call \emph{center} of $m$ the value $c(m) \eqdef \sigma(m)/d(m)$.
\end{defin}

Let $m$ be deficient and $p$ a prime such that $(m,p)=1$ and $p < c(m)$. By Proposition~\ref{prop:deficient}, it turns out that $mp$ is abundant. However, it is not guaranteed to be primitive abundant. Consider for example $m=16$, with $c(m)=31$. If we take $p = 7$, we have that $16 \cdot 7$ is abundant but not primitive abundant, since $8 \cdot 7$ is abundant, too. Another example, in which all prime numbers occur with multiplicity one, is $m= 2 \cdot 13 \cdot 31 = 806$. Then $5 < c(m)< 6$. If we take $p=3$, then $mp$ is abundant but not primitive abundant, since $2 \cdot 3 \cdot 13$ is abundant.

\begin{prop}
	\label{prop:center}
	The center enjoys the following properties:
	\begin{enumerate}
		\item\label{prop:center:alt} $c(m) = \displaystyle\frac{2m}{d(m)} - 1 =\frac{1}{\displaystyle\frac{2}{\sigmam(m)} - 1}$, for any deficient $m \in \NN$;
		\item\label{prop:center:mult} if $n > 1$ and $mn$ is deficient, then $c(mn) > c(m)$;
		\item\label{prop:center:inc} for any prime $p$, $c(p^e)$ is increasing in $e$ and $\displaystyle\lim_{e \rightarrow +\infty} c(p^e) = \frac{p}{p-2}$;
		\item\label{prop:center:dec} if $m$ is deficient and $p,q$ are primes coprime with $m$, $q>p>c(m)$, then $c(mq)<c(mp)$.
	\end{enumerate}
\end{prop}
\begin{proof}
	For \eqref{prop:center:alt}, we have
	\[
	c(m) = \frac{\sigma(m)}{2m-\sigma(m)} = \frac{\sigma(m)-2m+2m}{2m-\sigma(m)} =
	\frac{2m}{d(m)} -1
	\]
	Moreover
	\[
	c(m) = \frac{1}{\displaystyle\frac{2m - \sigma(m)}{\sigma(m)}} = \frac{1}{\displaystyle\frac{2}{\sigmam(m)} - 1}
	\]
	If we restrict ourselves to deficient numbers, ensuring $\sigmam(m) < 2$, we have that $c(m)$ is increasing with $\sigmam(m)$. Since  $\sigmam(mn) > \sigmam(m)$, under the hypothesis of this proposition, we have $c(mn) > c(m)$, proving \eqref{prop:center:mult}.
	From
	\[
	\sigmam(p^e) = \frac{1-\displaystyle\left(\frac{1}{p}\right)^{e+1}}{1-\displaystyle\frac{1}{p}},\qquad
	\lim_{e \rightarrow +\infty} \sigmam(p^e) = \frac{1}{1-\displaystyle\frac{1}{p}} = \frac{p}{p-1}
	\]
	we obtain that $c(p^e)$ is increasing in $e$ and
	\[
	\lim_{e \rightarrow +\infty} c(p^e) = \frac{1}{\displaystyle\frac{2(p-1)}{p} - 1}  = \frac{p}{p-2}
	\]
	thus proving \eqref{prop:center:inc}. To prove\eqref{prop:center:dec}, note that the hypothesis implies that $d(mp)$ and $d(mq)$ are positive, by Proposition~\ref{prop:deficient}, and conclude by an algebraic manipulation of $c(mq)<c(mp)$.
\end{proof}

We want to give appropriate conditions ensuring that $mp$ is primitive abundant. We know from Corollary~\ref{cor:primdiv} and Proposition~\ref{prop:deficient}, that a necessary condition for $mp^e$ to be primitive abundant is $\frac{p^e}{\sigma(p^{e-1})} > c(m/q)$ for each prime $q \div m$. Since our aim is to implement a program to enumerate PAN (see Section~\ref{sec:abundant-enumerate}), we would like to reduce the number of tests we need to perform each time. The following will be useful.

\begin{thm}[Structure Theorem for PAN]
	\label{thm:primitive-weird}
	Let $m$ be a deficient number, $e \in \NN$ and $p$ a prime such that $(m,p)=1$. Then $mp^e$ is primitive abundant iff $p^e/\sigma(p^{e-1}) < c(m)$,  $p^e/\sigma(p^{e-1}) > \displaystyle\frac{\sigma(m)}{d(m) + \displaystyle\frac{2m}{\sigma(q^\alpha)-1}}$ for each $q^\alpha \entdiv m$, and either $e=1$ or $p^{e-1}/\sigma(p^{e-2}) > c(m)$.
\end{thm}
\begin{proof}
	When $e>1$, by Corollary~\ref{cor:primdiv} and Proposition~\ref{prop:deficient}, we have that $mp^e$ is primitive abundant iff $p^e/\sigma(p^{e-1}) < c(m)$, $p^{e-1}/\sigma(p^{e-2}) > c(m)$ and $p^e / \sigma(p^{e-1}) > c(m/q)$ for each prime $q \div m$. If $e=1$, we have a similar result without the second condition. We prove that, if $q^\alpha \entdiv m$, then $c(m/q) = \displaystyle\frac{\sigma(m)}{d(m) + \displaystyle\frac{2m}{\sigma(q^\alpha)-1}}$. Let $\beta = \sigma(q^\alpha)/\sigma(q^{\alpha-1}) = \displaystyle\frac{1+ \dots + q^{\alpha}}{1+ \dots + q^{\alpha-1}}$. We have:
	\begin{multline}
	\label{eq:prop-primitive}
	c(m/q)=\frac{\sigma(m/q)}{d(m/q)} = \frac{\sigma(m)/\beta}{2m/q - \sigma(m)/\beta} = \frac{\sigma(m)}{2m\beta/q - \sigma(m)} = \\
	\frac{\sigma(m)}{\displaystyle 2m \frac{1+ \dots + q^{\alpha}}{q+ \dots + q^{\alpha}} - \sigma(m) } =  
	\frac{\sigma(m)}{\displaystyle 2m \left(1 + \frac{1}{q+ \dots + q^{\alpha}}\right) - \sigma(m)} = \frac{\sigma(m)}{d(m) + \displaystyle\frac{2m}{\sigma(q^\alpha)-1}}
	\end{multline}
	This concludes the proof.
\end{proof}

Since the expression on the r.h.s.~of \eqref{eq:prop-primitive} is increasing on $\sigma(q^\alpha)$, we can just keep track of the largest $\sigma(q^\alpha)$ of all the $q^\alpha$'s entirely dividing $m$. For computational reasons, the following variant of  \eqref{eq:prop-primitive} might be more efficient, since it only involves integer numbers:
\begin{equation}
\label{eq:cmq}
c(m/q) =   \frac{\sigma(m) - \displaystyle\frac{\sigma(m)}{\sigma(q^\alpha)}}{d(m) + \displaystyle\frac{\sigma(m)}{\sigma(q^\alpha)}}
\end{equation}

The following corollary has been already proved in \cite{AHMPWN}. We give here a different proof based on Theorem~\ref{thm:primitive-weird}.
\begin{cor}
	\label{cor:primitiveweird}
	If $m$ is deficient, $p$ is a prime such that $(m,p)=1$, $p < c(m)$ and $p  \geq \sigma(q^\alpha) - 1$ for each $q^\alpha \entdiv m$, then $mp$ is primitive abundant.
\end{cor}
\begin{proof}
	By Theorem~\ref{thm:primitive-weird}, $mp$ is primitive abundant whenever $ p > \displaystyle\frac{\sigma(m)}{d(m) + \displaystyle\frac{2m}{\sigma(q^\alpha)-1}}$ for each $q^\alpha \entdiv m$. We have
	\[
	\frac{\sigma(m)}{\displaystyle d(m) + \frac{2m}{\sigma(q^\alpha)-1}} = (\sigma(q^\alpha)-1) \frac{2m - d(m)}{d(m) (\sigma(q^\alpha)-1) + 2m} < \sigma(q^\alpha)-1\qedhere
	\] 
\end{proof}

\begin{remark}
	\label{rem:notprimitive}
	Due to the approximations in the previous proof, it is evident that the condition $p \geq \sigma(q^\alpha)-1$ is sufficient but not necessary. Consider $m=8$ and $p=7$. Although $7 < \sigma(8) -1$, it turns out that $8 \cdot 7$ is primitive abundant.
\end{remark}

The test for primitiveness in the case of square-free abundant numbers is particular easy, given the following:

\begin{cor}
	\label{cor:primitiveweird2}
	If $p_1 < \dots < p_k$ are primes such that $m = p_1 \cdots p_k$ is deficient, $p > p_k$ is a prime such that $mp$ is abundant, then $mp$ is primitive abundant.
\end{cor}
\begin{proof}
	Since $p > p_k$ then $(m,p)=1$. Moreover, for each $p_i$, we have $p_i \entdiv m$, hence $p \geq \sigma(p_i) -1 = p_i$. The thesis follows from  Corollary~\ref{cor:primitiveweird}.
\end{proof}

\subsection{Adding any prime factor to a deficient number}

We now consider the case when we start with a deficient number $m$ and add a prime factor $p$ with $p^\alpha \entdiv m$. We want to study under which conditions $mp$ is perfect, (primitive) abundant or deficient.

First of all, consider that Proposition~\ref{prop:deficient} does not hold when $(m,p) \neq 1$. For example, for $m=10=2 \cdot 5$ we have $c(m)=9$ but $2 \cdot 5^2$ is deficient. We may change Proposition~\ref{prop:deficient} in the following way:
\begin{prop}
	\label{prop:deficientbis}
	If $m$ is deficient and $p$ is a prime such that $p^\alpha \entdiv m$, then
	\begin{itemize}
		\item if $p\sigma(p^\alpha) <c(m)$ then $mp$ is abundant;
		\item if $p\sigma(p^\alpha) = c(m)$ then $mp$ is perfect;
		\item if $p\sigma(p^\alpha) > c(m)$ then $mp$ is deficient.
	\end{itemize}
\end{prop}
\begin{proof}
	We have that $d(mp)=2mp - \sigma(mp)= 2mp - \sigma(m) \frac{p^{\alpha+2}-1}{p^{\alpha+1}-1} = 2mp - \sigma(m)(p + \frac{p-1}{p^{\alpha+1}-1}) = d(m) p - \sigma(m) \frac{p-1}{p^{\alpha+1}-1} = d(m) p -\sigma(m) / \sigma(p^\alpha)$.
\end{proof}

\begin{example}
	If $m= 2 \cdot 5$, we have $c(m)=9$ but $5 \cdot \sigma(5)= 30$ hence $m \cdot 5$ is deficient. If $m = 2\cdot 5 \cdot 13 \cdot 61 \cdot 67$ we have $5651 < c(m) < 5652$. Since $61 \cdot \sigma(61) = 3782 < c(m)$, we have that $m \cdot 61$ is abundant.
\end{example}



We may also adapt Theorem~\ref{thm:primitive-weird} to the case when $p$ is not coprime with $m$ as follows:
\begin{thm}
	\label{thm:primitive-weird2}
	If $m$ is deficient and $p$ is a prime such that $p^\alpha \entdiv m$, we have that $mp$ is primitive abundant iff $p \sigma(p^\alpha) < c(m)$ and $p \sigma(p^\alpha) > \displaystyle\frac{\sigma(m)}{d(m) + \displaystyle\frac{2m}{\sigma(q^\beta)-1}}$ for each $q^\beta \entdiv m$ with $q \neq p$.
\end{thm}
\begin{proof}
	By Corollary~\ref{cor:primdiv} and Proposition~\ref{prop:deficientbis}, it is immediate that $mp$ is primitive abundant iff $p \sigma(p^\alpha) < c(m)$ and $p \sigma(p^\alpha) > c(m/q)$ for each prime $q \div m$ with $q \neq p$. In the proof of Theorem~\ref{thm:primitive-weird} we have shown that $c(m/q) = \displaystyle\frac{\sigma(m)}{d(m) + \displaystyle\frac{2m}{\sigma(q^\beta)-1}}$.
\end{proof}

We may repeat the same considerations we have made for Theorem~\ref{thm:primitive-weird}, regarding the fact we might only consider the largest $q^\beta \entdiv m$. Moreover, Equation~\ref{eq:cmq} still holds.

\section{Enumerating primitive abundant numbers}
\label{sec:abundant-enumerate}

%

Theorems~\ref{thm:primitive-weird} and \ref{thm:primitive-weird2} allow us to devise an algorithm for enumerating PAN or, more generally, primitive non-deficient numbers.
We will enumerate PAN on the basis of their factorization. For this reason, when $m=p_1^{e_1} \dotsm p_k^{e_k}$, we will always assume $p_1 < \dots < p_k$. Moreover, we will denote with $\omega(m) \eqdef k$ the number of distinct prime factors in $m$ and with $\Omega(m) \eqdef e_1 + \dots + e_k$ the number of prime factors in $m$ counted with their multiplicity.

Note that, if we fix the number of prime factors counted with multiplicity, then enumeration terminates, thanks to the following results.
\begin{lem}
    Given a number $m$ and $k \geq 0$, there are only finitely many PAN of the form $mn$ with $(m,n)=1$ and $\Omega(n)=k$.
\end{lem}
\begin{proof}
    By induction on $k$. For $k=0$ the result is trivial, either $m$ is primitive abundant and $n=1$ or it is not. If $k>1$, we distinguish whether $m$ is deficient or not. If  $m$ is not deficient, then $mn$ is never primitive abundant and the lemma holds. If $m$ is deficient, consider an $n$ such that $mn$ is primitive abundant and $\Omega(n)=k$. Then $n$ has the form $p_1^{e_1} \dotsm p_\ell^{e_\ell}$ with $p_1 < \cdots < p_\ell$ and $\sum_{i=1}^\ell e_i =k$. Since $mn$ is abundant, $\sigmam(n) > 2 / \sigmam(m)$. However, the abundancy of $n$ is bounded by
    \[
    \sigmam(n) = \sigmam(p_1^{e_1}) \dotsm \sigmam(p_\ell^{e_\ell}) \leq \sigmam(p_1^{e_1}) \dotsm \sigma_\ell(p_1^{e_\ell}) \leq (1/p_1+1)^k 
    \]
    Therefore, $(1/p_1+1)^k > 2/\sigmam(m)$, i.e., $1/p_1 > \sqrt[k]{2/\sigma(m)} - 1$. Since $m$ is deficient, $\sigmam(m) < 2$. Hence, the right hand side of this inequality is positive and $p_1$ is bounded from the above. Given one of the finitely many $p_1$ satisfying this condition and $e \in \{1, \dotsc, k\}$, by inductive hypothesis there are only finitely many $n'$ coprime with $mp^e$, with $\Omega(n')=k-e$ and such that $m p^e n$ is abundant. Varying $p$, these cover all possible values of $n$ in the statement of this lemma.
\end{proof}

\begin{thm}
    \label{thm:dick}
    For any $k>1$, there are only finitely many PAN $n$ with $\Omega(n)=k$.
\end{thm}
\begin{proof}
    This follows immediately from the previous lemma for $m=1$.
\end{proof}

We remark that Theorem~\ref{thm:dick} is a Corollary of~\cite{DicFOP,DicEAN} about finiteness of PAN with a fixed number of odd prime factors (counted without multiplicity) and a fixed power of 2.


\subsection{Square-free PAN}\label{sec:SFPANs}

We consider the special case of enumerating square-free PAN (SFPAN in the rest of the paper) with $k$ prime factors. Note that the more general case of primitive square-free non-deficient numbers is not interesting, since it is well-known that there is only one square-free perfect number which is 6.



The algorithm is a recursive procedure which takes a deficient number $m = \bar p_1 \dotsm \bar p_r$ with $\bar p_1 < \dots < \bar p_r$ and $r < k$ as input. Initially $m=1$. If $r < k-1$, for each prime $p  > c(m)$ we consider the number $\widetilde m = m p$, which is deficient by Proposition~\ref{prop:deficient}, and recursively call the procedure. If $r = k-1$, then we consider all primes $p$ contained in the possibly empty  open interval $(\bar p_r, c(m))$. By Corollary~\ref{cor:primitiveweird2}, each number of the form $m p$ is a PAN.

The algorithm needs a stopping condition in the case $r < k-1$, since we cannot actually test all the countably infinite primes $p > c(m)$. We decide to try primes in increasing order, stopping as soon as we find a $p$ such that there are no PAN starting with $m p$.  The complete description may be found in Algorithm~\ref{algo:sfpan}. The algorithm is easily checked to be correct. Completeness, i.e., the fact that the algorithm finds all SFPAN of the chosen form, will be discussed later.

\begin{algorithm}
	\LinesNumbered
	\SetKwFunction{sfpan}{sfpan}
	\SetKwData{count}{count}
	\SetKwData{innerCount}{innerCount}
	\Fn{\sfpan{k: nat, m: nat}}{
		\KwIn{$k$ is a natural number; $m = \bar p_1 \dotsm \bar p_r$  is a square-free deficient number, with $\bar p_1 < \dots < \bar p_r$ }
		\KwOut{all primitive abundant numbers of the form $m \cdot p_1 \dotsm  p_k$, with $\bar p_r < p_1 < \dots <  p_k$}
		\KwResult{the number of square-free primitive abundant numbers of the form above}
		\BlankLine
		\count $\leftarrow$ 0\;
		\eIf{$k=1$}{
			\ForEach{$p$ prime s.t. $\bar p_r < p < c(m)$}{
				\Print{mp}\;
				\count $\leftarrow$ \count + 1
			}
			\Return \count
		}{
			\ForEach{$p$ prime s.t. $p > \max(\bar p_r, c(m))$}{
				\innerCount $\leftarrow$ \sfpan{k-1, mp}\;
				\count $\leftarrow$ \count + \innerCount\;
				\If {$\innerCount = 0$} {
					\Return \count
				}
			}
		}
	}
	\caption{\label{algo:sfpan}Enumerating SFPAN with $k$ prime factors.}
\end{algorithm}

When we only want to count PAN, steps 3--6 of the algorithm may be replaced by a prime counting function. Using an implementation in SageMath of the algorithm and the prime counting function provided by Kim Walisch's \href{https://github.com/kimwalisch/primecount}{\texttt{primecount}} library, we managed to count the number of SFPAN  from $1$ up to $7$ distinct prime factors and odd SFPAN from $1$ up to $8$ distinct prime factors. The result is shown in Table~\ref{tab:sfpan-count} and form the OEIS sequences \href{https://oeis.org/A295369}{A295369} and \href{https://oeis.org/A287590}{A287590}.

We have also computed a list of SFPAN with up to 6 distinct prime factors, which is available on GitHub at \repourl.

\begin{table}[ht]
    \begin{tabular}{ccc}
        $\boldsymbol{\omega}$ & \textbf{\# all } & \textbf{\# odd}\\
        \hline
        1 & 0 & 0\\
        2 & 0 & 0\\
        3 & 1 & 0\\
        4 & 18 & 0\\
        5 & 610 & 87\\
        6 & 216054 & 14172\\
        7 & 12566567699 & 101053625\\
        8 & ? & 3475496953795289
    \end{tabular}
    \medskip
    \caption{\label{tab:sfpan-count}Number of SFPAN and odd SFPAN with given number of distinct prime factors.}
\end{table}

\subsection{Completeness of the enumeration algorithm}

The critical point of this algorithm is the stopping condition. Are we sure we do not loose any PAN? In order to ensure completeness of the search procedure, we need to prove that, if there is no SFPAN  $n$ such that $\omega(n)=k$ and whose factorization starts with $p_1 \dotsm p_r$, then there is no SFPAN $m$ with $\omega(m)=k$ and whose factorization starts with $p_1\dotsm p_{r-1}\cdot p$ for any $p > p_r$. We actually prove the contrapositive, i.e., that if  $p_1 \dotsm p_{r-1} \cdot p \cdot p_{r+1} \dotsm p_k$ is primitive abundant and $p_{r-1} < p_r < p$, then there exists an SFPAN $m$ with $\omega(m)=k$ and whose factorization starts with $p_1 \dotsm p_{r-1}\cdot p_r$. Note that $p_1 \dotsm p_{r-1}\cdot p_r\cdot p_{r+1} \dotsm p_k$ is abundant, but it might not be primitive abundant (see examples after Proposition~\ref{prop:ab1}).

Since a similar stopping condition will be used also in the algorithm of the next subsection, we will also consider the case of non-necessarily square-free PAN.

\begin{thm}[Deficient sequence completion]
	\label{thm:primitive-expand}
	If $m = \bar p_1^{e_1} \dotsm \bar p_r^{e_r}$ is deficient and $c(m) \geq \bar p_r$, then there are primes $p,  q$ with $\bar p_r < p < q$ such that $mp$ is deficient and $mpq$ is abundant. If $m$ is square-free, $mpq$ is primitive abundant.
\end{thm}

\begin{proof}
	First of all, when $m$ is square-free, $mpq$ is primitive abundant by Corollary~\ref{cor:primitiveweird2}.
	
	For the main part of the theorem, we consider initially the case $c(m) \geq 8$.
	Let $p$ be the smallest prime larger than $c(m)$. Then $p > c(m) \geq \bar p_r$, and $m p$ is deficient by Proposition~\ref{prop:deficient}. We need to find a prime $q > p$ such that $mpq$ is abundant. This requires $q < c(mp)$. We have
	\[
	c(mp) = \frac{\sigma(m)(p+1)}{2mp - \sigma(m)(p+1)} = \frac{\sigma(m)(p+1)}{d(m) p - \sigma(m)} =
	\frac{p+1}{\displaystyle\frac{p}{c(m)}  - 1}
	\]
	In 1952, Jitsuro Nagura \cite{NagICL} proved that for any $x \geq 8$ there is always a prime strictly between $x$ and $3x/2$. Therefore, by definition of $p$, using $x=c(m)$ in Nagura's Theorem, we have $p < 3 c(m)/2$ and
    \[
	\label{eq:cmp}
	c(m p) > 2(p + 1) = 2 p + 2
	\]\
	Again by Nagura's Theorem (or even weaker results), there is a prime $q$ in the interval $(p, 2p + 2)$. Thus, $q<2p+2<c(mp)$, and this concludes the case $c(m) \geq 8$.
	
	We now consider the case $c(m) < 8$, which implies $\bar p_r < 8$.
	
    
	If $\bar p_r=7$, then $c(m) \geq 7$. Let us take $p=11$ and $q=13$. Then $mp$ is deficient and $c(mp) = 12 / (11/c(m) -1 ) \geq 12 / (11/7 -1) = 21$, hence $mpq$ is abundant.
	
	If $\bar p_r = 5$, then $r \neq 1$, because $c(5^{e_r}) < 5/(5-2) = 5/3$ for any $e_r$ by \eqref{prop:center:inc} in Proposition~\ref{prop:center}. If both $2$ and $3$ are other factors of $m$, then $m$ is abundant by Proposition~\ref{prop:abundance}, because $2\cdot 3\cdot 5$ is abundant. Therefore, $m$ is either of the form $2^{e_1}5^{e_2}$ or $3^{e_1}5^{e_2}$ for $e_1, e_2 \geq 1$. Since $c(2 \cdot 5)=9 > 8$, we only consider the case $m=3^{e_1}5^{e_2}$, by \eqref{prop:center:mult} in Proposition~\ref{prop:center}. Since $c(3^2 \cdot 5^2) > 8$, the only cases remaining are: $m=3 \cdot 5$,  $m=3^2 \cdot 5$ and $m=3 \cdot 5^2$. However, $c(3 \cdot 5)=4 < 5$ and $c(3 \cdot 5^2)=62/13 < 5$. Therefore, the only $m$ satisfying the hypothesis of the theorem is $3^2 \cdot 5$, for which we may take $p=7$ and $q=11$.
	
	If $\bar p_r = 3$, then $r=1$: if $2$ also appears as a prime factor in $m$, then $m$ cannot be deficient since $2 \cdot 3$ is perfect, see Proposition~\ref{prop:abundance}. Then $m=3^{e_1}$ for some $e_1$. However, $c(3^{e_1}) < 3/(3-2) = 3$, hence $m$ does not satisfy the hypothesis of the theorem.
	
	If $\bar p_r = 2$, then $m=2^{e_r}$ and $\sigma(m)= c(m)= 2^{e_r+1} -1$. If $e_r \geq 3$ then $c(m) \geq 8$ and we fall into the previous case. For the remaining cases: if $m=2$, take $p=5$ and $q=7$; if $m=4$, take $p=11$ and $q=13$.
    \end{proof}

\begin{remark}\label{rem:nompq}
In the hypothesis of the previous theorem, when $m$ is not square-free, it might not be possible to obtain $p, q$ such that $mpq$ is primitive abundant. Consider $m=3^8 \cdot 5$, so that $8 < c(m) < 9$. If we determine $p$ as the smallest prime $p>c(m)$ and $q$ as the largest $q<c(mp)$ as in Proposition~\ref{prop:deficient}, we get $p=11, q=53$ and $m \cdot 11 \cdot 53$ which is abundant but not primitive abundant, since $3^7 \cdot 5 \cdot  11 \cdot 53$ is abundant, too. If we replace $53$ with smaller primes $q$ the abundance increases, because in general $\Delta(mq)-\Delta(mq')=\Delta(m)(q-q')$ whenever $q,q'$ are coprime with $m$, hence $m\cdot 11\cdot q$ is abundant and the number we obtain cannot be primitive abundant by Proposition~\ref{prop:ab1}. By increasing $p$ and computing the corresponding largest possible $q<c(mp)$, we get $m \cdot 13 \cdot 31$ and $m \cdot 17 \cdot 19$, but none of them is primitive abundant. We have $[c(m\cdot 19)]=17$, hence for primes $p\ge 19$ we get $c(m\cdot p)\ge c(m\cdot 19)>17$ by \eqref{prop:center:dec} in Proposition~\ref{prop:center}, and we have no primes $q>p$ making $mpq$ abundant.

Even relaxing the condition $\bar  p_r < p < q$ into $\bar p_r \leq p \leq q$, we do not get any PAN of the form $mpq$. Actually, $m \cdot 11^2$ is deficient, hence no number of the form $mpq$ is abundant when $p = q \ge 11$, by Proposition~\ref{prop:ab1}. If we take $p=5$, we have $13 < c(m \cdot 5) < 14$. Hence, $m \cdot 5 \cdot 13$ is abundant, but not primitive abundant, since $3^7 \cdot 5^2 \cdot 13$ is abundant, too. Finally, $m \cdot 5^2$ is not abundant.
\end{remark}

\begin{cor}
	\label{cor:primitive-expand}
	If $m =\bar p_1^{e_1} \dotsm \bar p_r^{e_r}$ is deficient and there exists a prime $p > \bar p_r$ such that $mp$ is abundant, then for each $s > 0$ there are primes $p_1 < \dots < p_s$ such that $p_1 > \bar p_r$, $m p_1 \dotsm p_s$ is abundant and $mp_1 \dotsm p_i$ is deficient for each $i < s$. Moreover, if $m$ is square-free, then $mp_1 \dotsm p_s$ is primitive abundant.
\end{cor}
\begin{proof}
	For $s=1$ the result follows by choosing $p_1 = p$. For $s > 1$, it follows by repeatedly applying Theorem~\ref{thm:primitive-expand}. Note that since $mp$ is abundant and $p > \bar p_r$, then $c(m) > p > \bar p_r$ by Proposition~\ref{prop:deficient}, hence the hypothesis of Theorem~\ref{thm:primitive-expand} hold and they are preserved by repeated applications. The result for $m$ square-free follows from Corollary~\ref{cor:primitiveweird2}.
\end{proof}

If $m$ is not square-free, the fact that $p_1,\dots, p_s$ may be chosen in such a way that $p_1,\dots p_s$ is primitive abundant is not always true: take for instance $m=3^8\cdot 5$ as in Remark~\ref{rem:nompq}. Then $m \cdot 7$ is abundant, but we have seen there are no $p$, $q$ such that $5 \leq p \leq q$ and $mpq$ is primitive abundant. 

The following theorem proves that the algorithm enumerating SFPAN is complete.
\begin{thm}
	\label{thm:complete}
	Let $m=p_1 \dotsm p_k$ be an abundant number, with $p_1 < \dots <p_k$. Let $j < k$ and  $p_{j-1} < \tilde p_j < p_j$ such that $p_1 \dotsm p_{j-1} \tilde p_j$ is deficient. Then, there are primes $\tilde p_{j+1} < \dots < \tilde p_k$ such that $\tilde p_j< \tilde p_{j+1}$, $\widetilde m = p_1 \dotsm p_{j-1} \tilde p_j \dotsm \tilde p_k$ is primitive abundant and $p_1 \dotsm p_{j-1} \tilde p_j \dotsm \tilde p_i$ is deficient for every $i < k$.
\end{thm}
\begin{proof}
	Let $r$ be the first index such that $p_1 \dotsm p_{j-1} \tilde p_{j} p_{j+1} \dotsm p_{r}$ is abundant. Then $r>j$ by hypothesis, and $r\le k$ because $m \tilde p_j / p_j$ is abundant by Proposition~\ref{prop:ab1}. Then we just apply Corollary~\ref{cor:primitive-expand} in order to add $k-r+1$ prime factors to  $p_1 \dotsm p_{j-1} \tilde p_{j} p_{j+1} \dotsm p_{r-1}$.
\end{proof}


\subsection{Possibly non square-free PAN}\label{sec:nonSFPANs}
An extension of the algorithm to find (non necessarily square-free) PAN $n$ with a fixed $\Omega(n)$ may be devised by allowing consecutive primes to be equal.

In other words, we see a number $m$ as the product of primes $\bar p_1 \cdots \bar p_r$ with $\bar p_1 \leq \dots \leq \bar p_r$. When called with $r < k-1$, the recursive procedure tries to extend $m$ to a deficient number $\widetilde m = mp$ using either $p=\bar p_r$ or $p > c(m)$ as for the square-free case. When $r=k-1$, the procedure tries to obtain an abundant number $mp$ by choosing either $p=\bar p_r$ or $p < c(m)$. In both cases, when $p=\bar p_r$, Proposition~\ref{prop:deficientbis} is used to decide whether $mp$ is abundant or deficient.

%

In the square-free case, when $r=k-1$, it is enough to choose $p > \bar p_r$ in order to ensure that $m p$ is not only abundant, but also primitive abundant. In the non square-free case this is not enough: we need to use a different lower bound for $p$, which can be computed using Theorem~\ref{thm:primitive-weird2}.

Another difference with respect to the square-free case is the stopping condition. The reason lies in the extension of Theorem~\ref{thm:complete} to possibly non square-free number.
\begin{thm}
    \label{thm:complete-nonsquarefree}
    Let $m=p_1 \cdots p_k$ be a PAN, with $p_1 \leq \dots \leq p_k$. Let $j < k$ and  $p_{j-1} < \tilde p_j < p_j$ such that $p_1 \cdots p_{j-1} \tilde p_j$ is deficient. Then, there are primes $\tilde p_{j+1} \leq \dots \leq \tilde p_k$ such that $\tilde p_j \leq \tilde p_{j+1}$,  $\widetilde m = p_1 \cdots p_{j-1} \tilde p_j \cdots \tilde p_k$ is abundant and $p_1 \cdots p_{j-1} \tilde p_j \cdots \tilde p_i$ is deficient for every $i < k$.
\end{thm}
\begin{proof}
    Since $\sigmam$ is sub-multiplicative, if we replace in $m$ the prime $p_j$ with $\tilde p_j$, the resulting number $m \tilde p_j / p_j$ is abundant. Actually $2 < \sigmam(m) = \sigmam(m p_j / p_j) \leq \sigmam(m/p_j) \sigmam(p_j) \leq \sigmam(m/p_j) \sigmam(\tilde p_j) = \sigmam(m \tilde p_j /p_j)$. Let $r$ be the first index (which by hypothesis is strictly  larger than $j$) such that $p_1 \dots p_{j-1} \tilde p_{j} p_{j+1} \cdots p_{r}$ is abundant. Then we just apply Corollary~\ref{cor:primitive-expand} in order to add $k-r+1$ prime factors to  $p_1 \dots p_{j-1} \tilde p_{j} p_{j+1} \cdots p_{r-1}$.
\end{proof}
We cannot guarantee that $\widetilde m$ is primitive abundant. For example, although $3^6 \cdot 5 \cdot 13 \cdot 31$ is primitive abundant, there is no $p \geq 11$ such that $m = 3^6 \cdot 5 \cdot 11 \cdot p$ is primitive abundant.

Since Theorem~\ref{thm:complete-nonsquarefree} does not ensure that $\widetilde{m}$ is primitive abundant, the procedure should return a boolean saying whether an abundant number (not necessarily a primitive abundant number) has been found, and stop when the recursive call returns false.

The complete description may be found in Algorithm~\ref{algo:pan_bigomega}.Using an implementation in SageMath of the algorithm we managed to count the number of PAN with from $1$ to $7$ prime factors (counted with their multiplicity) and odd SFPAN with from $1$ to $8$ prime factors. The results are shown in Table~\ref{tab:pan-count} and form the OEIS sequences \href{https://oeis.org/A298157}{A298157} and \href{https://oeis.org/A287728}{A287728}. We have also computed a list of PAN with up to 6 prime factors, which is available on GitHub at \repourl.

\begin{table}[ht]
    \begin{tabular}{ccc}
        $\boldsymbol\Omega$ & \textbf{\# all } & \textbf{\# odd}\\
        \hline
        1 & 0 & 0\\
        2 & 0 & 0\\
        3 & 2 & 0\\
        4 & 25	 & 0\\
        5 & 906 & 121\\
        6 & 265602 & 15772\\
        7 & 13232731828 & 102896101\\
        8 & ? & 3475842606319962
    \end{tabular}
    \medskip
    \caption{\label{tab:pan-count}Number of PAN and odd PAN with given number of prime factors counted with multiplicity.}
\end{table}

\begin{algorithm}
	\SetKwFunction{pndn}{pndn}
	\LinesNumbered
	\SetKw{And}{and}
	\SetKw{Or}{or}
	\SetKwFunction{sfpan}{sfpan}
	\SetKwData{count}{count}
	\SetKwData{found}{found}
	\SetKwData{lowerbound}{lowerbound}
	\SetKwData{innerCount}{innerCount}
	\SetKwData{innerFound}{innerFound}
	\Fn{\pndn{k: nat, m: nat = 1}}{
		\KwIn{$k$ is a natural number; $m = \bar p_1^{e_1} \dotsm \bar p_r^{e_r}$  is a deficient number with $\bar p_1 < \dots < \bar p_r$}
		\KwOut{all primitive non-deficient numbers of the form $m \cdot  p_1 \dotsm  p_k$, with $\bar p_r \leq  p_1 \leq \dots \leq p_k$.}
		\KwResult{a pair (\count, \found) where \count is the number of primitive non-deficient number of the form above, and \found is a boolean which is true when a (possibly non-primitive) non-deficient number of the form above has been found.}
		\Begin{
			\count $\leftarrow$ 0\;
			\found $\leftarrow$ false\;
			\eIf{$j=1$}{
				\If{there is a prime $p$ s.t. $\bar p_r < p \leq c(m)$}{
					\found $\leftarrow$ true\;
					\lowerbound $\leftarrow \max \{ c(m/p) \mid p \text{ is a divisor of } m\}$\;
					\ForEach{$p$ prime s.t. $\max(\bar p_r,\lowerbound) < p \leq c(m)$}{
						\Print{mp}\;
						\count $\leftarrow$ \count + 1
					}
				}
				\If {$\bar p_r \cdot \sigma(\bar p_r^{e_r}) \leq c(m)$}{
					\found $\leftarrow$ true\;
					\lowerbound $\leftarrow \max \{ c(m/p) \mid p < \bar p_r \text{ is a divisor of } m\}$\;
					\If {$\bar p_r \cdot \sigma(\bar p_r^{e^r}) > \lowerbound$}{
						\Print($m\bar p_r$)\;
						\count $\leftarrow$ \count +1
					}
				}
				\Return (\count,\found)
			}{
				\If{$m\bar p_r$ is deficient}{
					(\innerCount, \innerFound) $\leftarrow$ \pndn{$k-1$, $m\bar p_r$}\;
					\count $\leftarrow$ \count + \innerCount\;
					\found $\leftarrow$ \found \Or \innerFound
				}
				\ForEach{$p$ prime s.t. $p > \max(\bar p_r, c(m))$}{
					(\innerCount, \innerFound) $\leftarrow$ \pndn{$k-1$, $mp$}\;
					\count $\leftarrow$ \count + \innerCount\;
					\found $\leftarrow$ \found \Or \innerFound\;
					\If {\innerFound $=$ false} {
						\Return (\count, \found)
					}
				}
			}
		}
	}
	\caption{\label{algo:pan_bigomega}Enumerating primitive non-deficient numbers with $k$ prime factors, counted with their multiplicity.}
\end{algorithm}

\section{Weird numbers}\label{sec:weird}

In a previous paper \cite{AHMPWN}, we developed search algorithms which allowed us to find primitive weird numbers (PWN) with up to 6 different prime factors. However, we were not able to proceed further, because of the computational complexity involved. It was clear that a different approach was needed, which was suggested to us by the following known result.

\begin{prop}
    A number is primitive weird iff it is weird and primitive abundant.
\end{prop}
\begin{proof}
    If $n$ is primitive weird, by definition it is weird and abundant. We prove that, for any $m \div n$, $m$ is deficient.
    Assume $n=mk$. For the sake of contradiction, assume $m$ is non-deficient. Since $m$ cannot be weird by hypothesis, there is a subset $S$ of divisors of $m$ such that $m = \sum_{d \in S} d$. If $d$ is a divisor of $m$, $dk$ is a divisor of $n$. Hence $n = mk = \sum_{d \in S} dk$ is not weird, contradicting our hypothesis.

    On the other side, let $n$ be weird and primitive abundant. If $m \div n$ then $m$ is deficient, hence it cannot be weird. Therefore, $n$ is primitive weird.
\end{proof}

Given that PWN are only a particular case of PAN, we use the algorithms for enumerating PAN shown in the previous section, and add a straightforward check for weirdness, transforming them into algorithms for enumerating PWN.

Checking for weirdness can actually be made more efficient using the following well-known fact.
\begin{prop}
    \label{prop:weird-abundance}
    An abundant number $n$ is weird iff $\Delta(n)$ cannot be expressed as a sum of distinct proper divisors of $n$.
\end{prop}
\begin{proof}
    For a proof one can see, for instance, \cite[Lemma 2]{MelCIP}.
\end{proof}

\subsection{The square-free case}

We consider again Algorithm~\ref{algo:sfpan} for the square-free case. Since we are interested in finding PWN with several prime factors, and since it is not computationally feasible to enumerate all PAN in such cases, we provide as an additional input to the algorithm an \emph{amplitude} value $a$. At each step of the procedure, when iterating over primes larger than $c(m)$ (or smaller then $c(m)$ in the case $r=k-1$), we only consider at most the first $a$ primes.


Another generalization consists in starting the search procedure from a possibly non square-free deficient number $m$. This means that, in the Algorithm~\ref{algo:sfpan}, each $\bar p_i$ may be a power of a prime number, although new primes added by the procedure are always square-free. However, when $r=k-1$, we only consider primes $p$ which are larger than $\sigma(q^\alpha)$ for each $q^\alpha \entdiv m$. In such a way, by Corollary~\ref{cor:primitiveweird}, the abundant numbers found by the search procedure turns out to be primitive abundant. When $m$ is a power of $2$, $c(m)=\sigma(m)$ and there are no additional constraints on the choice of the last prime.


\begin{remark}
	In determining whether a number is weird, the sufficient conditions in Theorem~3.1 of our previous paper \cite{AHMPWN} could be employed. However, experimental evaluation has shown that most of the weird numbers generated with our approach fail to satisfy these conditions. Therefore, a direct proof of weirdness using Proposition~\ref{prop:weird-abundance} is employed.
\end{remark}

%
%
%

The weird numbers generated by this procedure tend to be huge. At each step, since we choose $p$ close to $c(m)$, we minimize the deficiency of $\widetilde m = mp$. However, when recursively calling the search procedure on $\widetilde m$, since $d(\widetilde{m})$ is small, $c(\widetilde m)$ is quite large. This is repeated step after step, leading to very large prime factors.
%
For example, all the PWN we have generated with $\omega=12$ are larger than $10^{900}$. Since dealing with these huge numbers is cumbersome, we represent them in a form we have called \emph{index sequence}, that turned out to be very useful.

\begin{defin}[Index sequence]
\label{def:index}
Given a number $m = p_1^{e_1} \cdots p_k^{e_k}$, we define $\iota(m)$, the \emph{index sequence} associated to $m$,  as the sequence $[(\iota_1,e_1), \dots, (\iota_k,e_k)]$ with $\iota_1, \dots, \iota_k \in \mathbb{Z}$ such that:
\begin{itemize}
	\item if $\iota_i = 0$, then $p_i = c(w_{i-1})$;
	\item if $\iota_i > 0$, then $p_i$ is the $\iota_i$-th prime larger than $c(w_{i-1})$;
	\item if $\iota_i < 0$, then $p_i$ is the $|\iota_i|$-th prime smaller than $c(w_{i-1})$,
\end{itemize}
where $w_i \eqdef \Pi_{j=1}^i p_j^{e_j}$. To ease notation, we write each pair $(\iota, e)$ as $\iota^e$ or just $\iota$ if $e=1$.
\end{defin}

For example, the number $m = 2^2 \cdot 13 \cdot 17 \cdot 443 \cdot 97919 \cdot 563915507$ is represented by the index sequence $[1^2, 2, 1, 1, 1, -2]$, because 2 is the 1st prime larger than $c(1)=1$, 13 is the 2nd prime larger than $c(2^2)=7$, 17 is the 1st prime larger than $c(2^2\cdot 13)=16.\bar3$, and so on. All index sequences generated by our search procedure have positive indices for all but the last position. All the indices have an absolute value smaller than the amplitude parameter $a$.

\begin{table}[t]
\centering
\small
\begin{tabular}{|l|l|l|l|}
$\boldsymbol{\omega}$ & \textbf{factored weird number} & $\boldsymbol{\Delta}\mathbf{(w)}$ & \textbf{index sequence}\\
\hline
3 & $2 \cdot 5 \cdot 7$ & 4 &  $[1, 1, -1]$ \\
3 & $4 \cdot 11 \cdot 19$ & 8 & $[1^2, 1, -1]$\\
3 & $8 \cdot 17 \cdot 127$ & 16 & $[1^3, 1, -2]$\\
3 & $8 \cdot 19 \cdot 71$ & 16 & $[1^3, 2, -2]$ \\
3 & $8 \cdot 19 \cdot 61$ & 56 & $[1^3, 2, -4]$\\
3 & $8 \cdot 23 \cdot 43$ & 16 & $[1^3, 3, -1]$\\
3 & $8 \cdot 29 \cdot 31$ & 16 & $[1^3, 4, -1]$\\
4 & $2 \cdot 5 \cdot 11 \cdot 53$ & 4 & $[1, 1, 1, -1]$ \\
4 & $2 \cdot 5 \cdot 13 \cdot 31$ & 4 & $[1, 1, 2, -1]$ \\
4 & $4 \cdot 11 \cdot 23 \cdot 251$ & 8 & $[1^2, 1, 1, -1]$\\
4 & $4 \cdot 11 \cdot 23 \cdot 241$ & 88 & $[1^2, 1, 1, -2]$\\
4 & $4 \cdot 11 \cdot 31 \cdot 67$ & 8 & $[1^2, 1, 3, -1]$\\
4 & $4 \cdot 13 \cdot 17 \cdot 439$ & 8 & $[1^2, 2, 1, -1]$\\
4 & $8 \cdot 17 \cdot 137 \cdot 9311$ & 16 & $[1^3, 1, 1, -1]$ \\
4 & $8 \cdot 17 \cdot 139 \cdot 4723$ & 16 & $[1^3, 1, 2, -1]$\\
4 & $8 \cdot 19 \cdot 79 \cdot 1499$ & 16 & $[1^3, 2, 1, -1]$\\
4 & $8 \cdot 19 \cdot 83 \cdot 787$ & 16 & $[1^3, 2, 2, -1]$\\
4 & $8 \cdot 23 \cdot 67 \cdot 139$ & 16 & $[1^3, 3, 5, -1]$\\
5 & $2 \cdot 5 \cdot 11 \cdot 59 \cdot 647$ & 20 & $[1, 1, 1, 1, -1]$ \\
5 & $4 \cdot 11 \cdot 23 \cdot 257 \cdot 13003$ & 8 & $[1^2, 1, 1, 1, -1]$ \\
5 & $4 \cdot 11 \cdot 23 \cdot 257 \cdot 13001$ & 88 & $[1^2, 1, 1, 1, -2]$ \\
5 & $4 \cdot 11 \cdot 23 \cdot 263 \cdot 6047$ & 88 & $[1^2, 1, 1, 2, -1]$ \\
5 & $4 \cdot 13 \cdot 17 \cdot 449 \cdot 24799$ & 232 & $[1^2, 2, 1, 2, -1]$ \\
5 & $4 \cdot 13 \cdot 23 \cdot 61 \cdot 1657$ & 8 & $[1^2, 2, 3, 2, -1]$\\
5 & $8 \cdot 17 \cdot 137 \cdot 9337 \cdot 3953791$ & 272 & $[1^3, 1, 1, 3, -1]$\\
5 & $8 \cdot 17 \cdot 137 \cdot 9341 \cdot 3346951$ &  16 & $[1^3, 1, 1, 4, -1]$  \\
5 & $8 \cdot 17 \cdot 137 \cdot 9341 \cdot 3346883$ & 7088 & $[1^3, 1, 1, 4, -6]$\\
5 & $8 \cdot 23 \cdot 47 \cdot 1091 \cdot 107209$ & 976 & $[1^3, 3, 1, 2, -1]$\\
5 & $8 \cdot 23 \cdot 47 \cdot 1103 \cdot 51839$ &  368 & $[1^3, 3, 1, 5, -1]$ \\
5 & $8 \cdot 23 \cdot 71 \cdot 127 \cdot 6689$ & 16 &  $[1^3, 3, 6, 1, -1]$ \\
5 & $8 \cdot 31 \cdot 37 \cdot 163 \cdot 186959$ & 16 & $[1^3, 5, 2, 1, -1]$\\
5 & $8 \cdot 37 \cdot 43 \cdot 67 \cdot 15227$ & 16 & $[1^3, 6, 5, 1, -1]$ \\
6 & $4 \cdot 11 \cdot 23 \cdot 269 \cdot 4003 \cdot 24766559$ & 88  & $[1^2, 1, 1, 3, 1, -1]$\\
6 & $4 \cdot 11 \cdot 23 \cdot 269 \cdot 4013 \cdot 1508909$ & 248 & $[1^2, 1, 1, 3, 3, -1]$\\
6 & $4 \cdot 13 \cdot 17 \cdot 443 \cdot 97919 \cdot 563915507$ & 1768 & $[1^2, 2, 1, 1, 1, -2]$\\
6 & $4 \cdot 13 \cdot 17 \cdot 443 \cdot 97931 \cdot 330611657$ & 4888 & $[1^2, 2, 1, 1, 3, -1]$ \\
6 & $8 \cdot 17 \cdot 137 \cdot 9349 \cdot 2561627 \cdot 3280965162749$ & 272 & $[1^3, 1, 1, 6, 1, -1]$\\
6 & $8 \cdot 17 \cdot 137 \cdot 9349 \cdot 2561651 \cdot 252384300173$ & 272 & $[1^3, 1, 1, 6, 3, -1]$ \\
6 & $8 \cdot 17 \cdot 139 \cdot 4783 \cdot 389749 \cdot 8454956717$ &  7088 &$[1^3, 1, 2, 5, 2, -1]$\\
6 & $8 \cdot 23 \cdot 47 \cdot 1087 \cdot 167863 \cdot 197246914559$ & 16 &$[1^3, 3, 1, 1, 1, -1]$  \\
7 & $2 \cdot 5 \cdot 11 \cdot 89 \cdot 167 \cdot 829 \cdot 7972687$ & 20 & $[1, 1, 1, 8, 6, 1, -1]$  \\
7 & $4 \cdot 13 \cdot 17 \cdot 443 \cdot 97919 \cdot 563915549 \cdot 10965542434977103$ & 1768 & $[1^2, 2, 1, 1, 1, 2, -1]$
\end{tabular}
\medskip
\caption{\label{table:smallweirds}Some PWN found by our search algorithm. The first column is the number of prime factors. For each $\omega$, entries are in lexicographic order with respect to the index sequence.}
\end{table}

\begin{table}[t]
\small
\centering
\begin{tabular}{|l|l||l|l|}
$\boldsymbol{\omega}$ & \textbf{index sequence} & $\boldsymbol{\omega}$ & \textbf{index sequence}\\
\hline
7 & $[1^3, 1, 2, 2, 1, 4, -1]$ & 11 & $[1^2, 2, 1, 1, 1, 1, 1, 1, 2, 1, -1]$ \\
7 & $[1^3, 6, 6, 1, 1, 6, -6]$ & 11 & $[1^2, 2, 1, 1, 1, 1, 1, 3, 1, 3, -1]$\\
7 & $[1^3, 6, 6, 1, 3, 2, -5]$ & 11 &  $[1^2, 2, 1, 1, 1, 2, 1, 2, 2, 2, -3]$\\
7 & $[1^3, 6, 6, 1, 3, 2, -6]$ & 11 &  $[1^2, 2, 1, 1, 1, 2, 2, 1, 1, 3, -2]$\\
7 & $[1^3, 6, 6, 1, 3, 5, -3]$ & 11 &  $[1^2, 2, 1, 1, 1, 2, 3, 2, 2, 1, -2]$\\
7 & $[1^3, 6, 6, 1, 4, 5, -1]$ & 11 & $[1^2, 2, 1, 1, 1, 3, 3, 3, 1, 3, -2]$ \\
7 & $[1^3, 6, 6, 1, 5, 2, -5]$ & 11 & $[1^2, 2, 1, 1, 2, 1, 1, 2, 2, 1, -1]$ \\
7 & $[1^3, 6, 6, 1, 5, 4, -2]$ & 11 & $[1^2, 2, 1, 1, 2, 2, 1, 3, 1, 1, -1]$\\
7 & $[1^3, 6, 6, 1, 6, 1, -6]$ & 11 & $[1^2, 2, 1, 1, 2, 3, 1, 1, 1, 1, -1]$\\
7 & $[1^3, 6, 6, 1, 6, 3, -2]$ & 11 & $[1^2, 2, 1, 1, 2, 3, 1, 1, 2, 3, -2]$\\
7 & $[1^3, 6, 6, 1, 6, 4, -3]$ & 11 &  $[1^2, 2, 1, 1, 3, 2, 2, 1, 2, 1, -3]$ \\
7 & $[1^3, 6, 6, 1, 6, 4, -4]$ & 11 & $[1^2, 2, 1, 2, 1, 2, 1, 1, 1, 3, -1]$\\
8 & $[1^2, 2, 1, 1, 1, 2, 1, -2]$ & 11 & $[1^2, 2, 1, 2, 1, 2, 1, 3, 2, 1, -1]$\\
8 & $[1^2, 2, 1, 1, 3, 3, 1, -2]$ & 11 & $[1^2, 2, 1, 2, 1, 3, 1, 1, 1, 1, -1]$ \\
8 & $[1^2, 2, 1, 2, 1, 3, 3, -3]$ & 12 & $[1^2, 2, 1, 1, 1, 2, 1,  1, 2, 2, 3, -1]$ \\
9 & $[1^2, 2, 1, 1, 1, 2, 1, 3, -3]$ & 12 & $[1^2, 2, 1, 1, 1, 2, 1, 2, 2, 1, 3, -3]$\\
9 & $[1^2, 2, 1, 1, 2, 3, 1, 2, -3]$ & 12 & $[1^2, 2, 1, 1, 1, 2, 2, 1, 1, 3, 1, -1]$\\
9 & $[1^2, 2, 1, 2, 1, 2, 3, 1, -3]$ & 12 & $[1^2, 2, 1, 1, 1, 2, 2, 1, 2, 3, 1, -1]$ \\
9 & $[1^2, 2, 1, 2, 1, 3, 3, 1, -1]$ & 12 & $[1^2, 2, 1, 1, 3, 1, 1, 3, 1, 1, 1, -3]$  \\
10 & $[1^2, 2, 1, 1, 1, 1, 2, 3, 2, -3]$ & 12 & $[1^2, 2, 1, 1, 3, 1, 1, 3, 1, 1, 3, -3]$\\
10 & $[1^2, 2, 1, 1, 1, 1, 3, 1, 3, -2]$ & 12 & $[1^2, 2, 1, 1, 3, 1, 2, 2, 1, 1, 2, -1]$\\
10 & $[1^2, 2, 1, 1, 1, 1, 3, 3, 1, -3]$ & 13 & $[1^2, 2, 1, 1, 1, 3, 3, 2, 2, 3, 3, 2, -2]$\\
10 & $[1^2, 2, 1, 1, 1, 2, 2, 1, 1, -1]$ & 13 & $[1^2, 2, 1, 2, 1, 1, 1, 1, 1, 2, 3, 2, -1]$ \\
10 & $[1^2, 2, 1, 1, 1, 2, 3, 1, 1, -1]$ & 13 & $[1^2, 2, 1, 2, 1, 1, 1, 1, 1, 3, 1, 1, -2]$ \\
10 & $[1^2, 2, 1, 1, 1, 3, 2, 1, 3, -3]$ & 13 & $[1^2, 2, 1, 2, 1, 1, 1, 1, 2, 1, 1, 2, -1]$\\
10 & $[1^2, 2, 1, 1, 2, 1, 1, 1, 2, -2]$ & 13 & $[1^2, 2, 1, 2, 1, 1, 1, 1, 2, 1, 3, 1, -2]$ \\
10 & $[1^2, 2, 1, 1, 2, 1, 1, 1, 3, -1]$ & 13 & $[1^2, 2, 1, 2, 1, 1, 1, 1, 2, 1, 3, 1, -2]$\\
10 & $[1^2, 2, 1, 1, 2, 1, 2, 3, 1, -3]$ & 13 & $[1^2, 2, 1, 2, 1, 1, 1, 2, 3, 1, 2, 1, -3]$ \\
10 & $[1^2, 2, 1, 1, 2, 3, 3, 1, 3, -2]$ & 13 & $[1^2, 2, 1, 2, 1, 1, 1, 3, 1, 3, 2, 3, -2]$\\
10 & $[1^2, 2, 1, 1, 2, 3, 3, 3, 1, -2]$ & 13 & $[1^2, 2, 1, 2, 1, 1, 1, 3, 1, 3, 3, 2, -3]$ \\
10 & $[1^2, 2, 1, 1, 3, 2, 2, 1, 3, -1]$ & 13 & $[1^2, 2, 1, 2, 1, 1, 1, 3, 2, 2, 1, 1, -1]$\\
10 & $[1^2, 2, 1, 1, 3, 2, 2, 2, 1, -2]$ & 13 & $[1^2, 2, 1, 2, 1, 1, 1, 3, 3, 3, 3, 1, -1]$ \\
10 & $[1^2, 2, 1, 1, 3, 2, 3, 1, 2, -2]$ & 13 & $[1^2, 2, 1, 2, 1, 1, 2, 3, 1, 2, 1, 1, -1]$  \\
10 & $[1^2, 2, 1, 1, 3, 3, 3, 2, 3, -2]$ & 13 & $[1^2, 2, 1, 2, 1, 1, 3, 1, 1, 1, 3, 1, -1]$\\
10 & $[1^2, 2, 1, 2, 1, 2, 2, 3, 2, -2]$ & 13 & $[1^2, 2, 1, 2, 1, 1, 3, 1, 3, 3, 1, 2, -1]$ \\
10 & $[1^2, 2, 1, 2, 1, 2, 2, 3, 2, -3]$ & 13 & $[1^2, 2, 1, 2, 1, 1, 3, 3, 1, 1, 1, 3, -3]$\\
10 & $[1^2, 2, 1, 2, 1, 3, 1, 3, 3, -3]$ & 13 & $[1^2  2, 1, 2, 1, 1, 3, 3, 2, 2, 1, 3, -1]$\\
10 & $[1^2, 2, 1, 2, 1, 3, 2, 3, 2, -1]$ & 13 & $[1^2, 2, 1, 2, 1, 1, 3, 3, 2, 2, 1, 3, -2]$\\
10 & $[1^2, 2, 1, 2, 1, 3, 2, 3, 2, -2]$ & 14 & $[1^2, 2, 1, 2, 1  1, 1, 3, 3, 2, 3, 1, 2, -2]$\\
10 & $[1^2, 2, 1, 2, 1, 3, 3, 2, 2, -3]$ & 15 & $[1^2, 2, 1, 2, 1, 1, 1, 1, 1, 1, 1, 1, 1, 1, -2]$  \\
10 & $[1^2, 2, 1, 2, 2, 3, 1, 1, 3, -1]$ & 16 & $[1^2, 2, 1, 2, 1, 1, 1, 1, 1, 2, 1, 3, 3, 1, 2, -1]$
\end{tabular}
\medskip
\caption{\label{table:bigweirds}Other PWN found with our search algorithm. Since these numbers are large, only the index sequence is shown. As an example, the first two entries are 54 and 37 digits long, while the last three entries are 3608, 7392 and 14712 digits long respectively. For each $\omega$, entries are in lexicographic order.}
\end{table}

\begin{remark}\label{rem:limitation}
	Having to deal with huge numbers is a limitation of our approach: increasing the value of $k$ has a big impact on performance because not only is the search space increased by a factor $a$ (the amplitude of the search space) but the numbers we deal with also become much larger. Experimentally we see that, when $p_{i}$ is near $c(p_1^{e_1} \cdots p_{i-1}^{e_{i-1}})$, then each prime is roughly double the size of the preceding one, in terms of the number of digits. Therefore, there is an exponential increase in the size of factors, which impacts all operations on these numbers, but particularly the procedure for determining the (pseudo-)prime immediately preceding or following a given number $n$. This procedure essentially works by repeatedly calling a (pseudo-)primality test with consecutive odd numbers until a new (pseudo-)prime is found. Since in the average the gap between primes is $\log n$ and the Baillie–PSW primality test \cite{BW,PSW} used by SageMath takes time proportional to $\log^3 n$, the computational complexity of determining the next prime is roughly $\log^4 n$, i.e., $4^k$. This makes it extremely hard to run our algorithms with values of $\Omega > 16$, even with a small value for the amplitude.


	On the other side, it seems that the abundant numbers $m$ generated in this way are very likely going to be weird. This is, at least in part, due to the fact that $\Delta(m)$ is low if compared to $m$ and its prime factors. A low abundance is unlikely to be expressible as sum of divisors of $m$, see Proposition~\ref{prop:weird-abundance}.
\end{remark}

In line with the previous remark, many PWN are easily found starting from a power of two for $m$ and a small amplitude for $a$. Tables~\ref{table:smallweirds} and~\ref{table:bigweirds} contain some of the PWN we have found starting from the following parameters:
\begin{itemize}
\item $m=2$, $a=8$, $k \in \{3, \dots, 10\}$;
\item $m=4$, $a=3$, $k \in \{3, \dots, 16\}$;
\item $m=8$, $a=6$, $k \in \{3, \dots, 10\}$.
\end{itemize}

Table~\ref{table:smallweirds} contains, for each PWN, both its factorization and its index sequence. Table~\ref{table:bigweirds} only contains index sequences since the constituent primes would not fit on the page.
In particular, we mention the following results:
\begin{itemize}
    \item We have found PWN with up to 16 distinct prime factors. Previously, PWN with 6 distinct prime factors were shown in \cite{AHMPWN}, and only one with 7 distinct prime factors was known~\cite{BREPCL}, while no PWN was known with $8$ or more distinct prime factors.
    \item The PWN with 16 distinct prime factors has 14712 digits. This is, to the best of our knowledge, the largest PWN known, the previously largest  having 5328 digits \cite{MelCIP}.
\end{itemize}

Note that, for the sake of efficiency, the search algorithm uses pseudo-primes. However, all the factors for the weird numbers in Tables~\ref{table:smallweirds}, \ref{table:bigweirds}, \ref{table:primesquare} and~\ref{table:primesquaremore} have been validated to be real primes, even using additional software such as \href{http://www.ellipsa.eu/index.html}{Primo} (a primality proving program based on the Elliptic Curve Primality Proving algorithm).

\begin{remark}\label{rem:mfour}
Explaining the fact we find so many PWN only on the basis of their abundance is not satisfactory. In particular, by looking at the tables, it is evident that the initial value $m=4$ is the best choice for determining PWN, at least for low values of the amplitude parameter: with a value of just $3$, we could find PWN with $k$ distinct prime factors for all $k$ between $3$ and $16$. The results for $m=2$ and $m=8$ were less satisfactory, even using much larger values for the parameter $a$. We will investigate this behavior in a forthcoming paper.
\end{remark}

\subsection{PWN with square factors}
\label{sect:square-factors}

Another weirdness in the realm of weirds is the rarity of PWN with odd prime factors of multiplicity greater than one. To the best of our knowledge, up to now there were only five known PWN with a square odd prime factor, listed in the OEIS sequence \href{https://oeis.org/A273815}{A273815}, and no PWN with an odd prime factor of multiplicity strictly greater than two is known.

Using an extension of Algorithm~\ref{algo:pan_bigomega} we have found hundreds of new PWN with at least one odd prime factor of multiplicity greater than one. A selection of them may be found in Table~\ref{table:primesquare}. We find that there are no such PWN for $\Omega < 7$, and the list for $\Omega=7$ is complete. From $\Omega = 8$ onwards, our list is only partial. None of the PWN we have found has odd prime factors with exponent greater than two. 

\begin{table}[t]
    \centering
    \small
    \begin{tabular}{|l|Hl|l|HH}
        $\boldsymbol{\Omega}$ & $\omega$  & \textbf{factored weird number} & \textbf{index sequence} & $\Delta(w)$ & new? \\
        \hline
        7 & 5 & $2^2 \cdot 13^2 \cdot 19 \cdot 383 \cdot 23203$ \textdagger& $[1^2, 2^2, 1, 2, -1]$ &8 & no\\
        7 & 5 & $2^2 \cdot 13 \cdot 17 \cdot 443^2 \cdot  194867$ & $[1^2, 2, 1, 1^2, -6]$ & 103192 & yes\\
        7 & 6 & $2 \cdot 5^2  \cdot 29 \cdot 37 \cdot 137 \cdot 211$ \textdagger& $[1, 1^2, 4, 3, 11, -1]$ &  20 & no\\
        7 & 6 & $2 \cdot 5 \cdot 11^2 \cdot 103 \cdot 877 \cdot 2376097$ & $[1, 1, 1^2, 3, 1, -1]$ & 4  & yes\\
        7 & 6 & $2 \cdot 5 \cdot 11 \cdot 127^2 \cdot 167 \cdot 223$ \textdagger& $[1, 1, 1, 15^2, 15, -1]$ &  4 & no\\
        8 & 5 & $2^3 \cdot 17^2 \cdot 277 \cdot 1979 \cdot 115259$ &  $[1^3, 1^2, 6, 4, -1]$ & 272 & yes\\
        8 & 5 & $2^3 \cdot 23^2 \cdot 53 \cdot 691 \cdot 32587$ & $[1^3, 3^2, 1, 3, -1]$ & 16 & yes\\
        8 & 6 & $2^2 \cdot 13 \cdot 17 \cdot 449 \cdot 24809^2 \cdot 351659387$ & $[1^2, 2, 1, 2, 1^2, -3]$ &123773096  & yes\\
        8 & 6 & $2^2 \cdot 13 \cdot 17 \cdot 449 \cdot 24809^2 \cdot 351659377$ & $[1^2, 2, 1, 2, 1^2, -4]$ & 137667016 & yes \\
        9 & 3 & $2^6 \cdot 137^2 \cdot 1931$ \textdagger&  $[1^6, 2^2, -1]$ &956 & no\\
        9 & 5 & $2^4 \cdot 37 \cdot 197 \cdot 58313^2 \cdot 3400230989$ &   $[1^4, 1, 1, 1^2, -4]$& 10729912 & yes\\
        9 & 5 & $2^4 \cdot 41 \cdot 131 \cdot 21517^2 \cdot 14007547$ & $[1^4, 2, 1, 6^2, -1]$ & 1726528& yes\\
        9 & 7 & $2^2 \cdot 13 \cdot  17 \cdot  443 \cdot  97919 \cdot  563915543^2 \cdot \ap{17}$ &  $[1^2, 2, 1, 1, 1, 1^2, -5]$ & 389552811218584 & yes\\
        9 & 7 & $2^2 \cdot 13 \cdot 17 \cdot 449 \cdot 24809^2 \cdot 351659531 \cdot \ap{16}$ &  $[1^2, 2, 1, 2, 1^2, 3, -1]$  & 136920152 & yes\\
        10 & 7 & $2^4 \cdot 41 \cdot 131 \cdot 21493 \cdot 46175611^2 \cdot \ap{14}$ & $[1^4, 2, 1, 3, 2^2, -5]$ & 1726528 & yes\\
        10 & 7 & $2^3 \cdot 37 \cdot 47 \cdot 59 \cdot 102607 \cdot 1503940237^2 \cdot \ap{17}$ &  $[1^3, 6, 6, 1, 1, 5^2, -6]$  & 5516082397305104 & yes\\
        11 & 3 & $2^8 \cdot 797^2 \cdot 1429$ \textdagger& $[1^8, 42^2, -1]$  & 678 & no\\
        12 & 5 & $2^7 \cdot 359 \cdot 883 \cdot 2535977^2 \cdot 6431171736581$ & $[1^7, 18, 1, 1^2, -1]$& 1785328832 & yes\\
        13 & 3 & $2^{10} \cdot 2081^2 \cdot 129083$ & $[1^{10}, 4^2, -1]$ & &\\
        15 & 3 & $2^{12} \cdot 9103^2 \cdot 81847$ & $[1^{12}, 101^2, -2]$ & & yes
    \end{tabular}
    \medskip
    \caption{\label{table:primesquare}Some of the PWN with square odd prime factors that we have found. PWN already listed in \href{https://oeis.org/A273815}{A273815} are marked with \textdagger. For $\Omega=7$, this is the complete list of all the PWN with at least one odd prime factor with exponent greater than one. $\ap{i}$ denotes a prime with $i$ digits. Entries are in lexicographic order of index sequences.}
\end{table}

On the other side we have found many PWN which have two odd prime factors with exponent greater than one, which were not known up to now. One of them is:
\begin{multline*}
2^2 \cdot 13 \cdot 17 \cdot 449 \cdot 24809 \cdot 223797481 \cdot 13437522702621389^2 \cdot \\ 3074438401877924358902212859897^2 \cdot\\ 144038537693729891876284023491399806504775375343886878276167
\end{multline*}
whose index sequence is
\[
[1^2, 2, 1, 2, 1, 1, 1^2, 1^2, -1]
\]
Other PWN with 2 square odd prime factors are given in Table~\ref{table:primesquaremore}. Actually, the last of them has 3 square odd prime factors, so it is likely that there are weird numbers with any number of square odd prime factors, provided $\Omega$ is big enough.

\begin{table}[t]
    \centering
    \small
    \begin{tabular}{|l|Hl|l|H}
        $\boldsymbol{\Omega}$ & $\omega$  & \textbf{factored weird number} & \textbf{index sequence} & $\Delta(w)$ \\
        \hline
        12 & & $w_6  \cdot 13437522702621389^2 \cdot \ap{31}^2 \cdot \ap{60}$ & $ [1^2, 2, 1, 2, 1, 1, 1^2, 1^2, -1]$\\
        12 & & $w_6 \cdot 13437522702621427^2 \cdot \ap{31}^2 \cdot \ap{60}$ & $[1^2, 2, 1, 2, 1, 1, 2^2, 1^2, -3]$\\
        12 & & $w'_6 \cdot 13826118575254057^2 \cdot \ap{32}^2 \cdot \ap{61}$ & $[1^2, 2, 1, 1, 1, 1, 1^2, 1^2, -4]$\\
        13 & & $w'_6 \cdot 13826118575254057^2 \cdot \ap{32} \cdot \ap{61}^2 \cdot \ap{118}$& $[1^2,  2, 1, 1, 1, 1, 1^2, 1, 4^2, -1]$\\
        14 & & $w'_6 \cdot 13826118575254057^2 \cdot \ap{32} \cdot \ap{60}^2 \cdot \ap{118} \cdot \ap{233}$ & $[1^2, 2, 1, 1, 1, 1, 1^2, 2, 1^2, 2, -1]$\\
        15 & & $w'_6 \cdot 13826118575254057 \cdot \ap{32}^2 \cdot \ap{61}^2 \cdot \ap{120}^2 \cdot \ap{237}$ & $[1^2, 2, 1, 1, 1, 1, 1, 1^2, 1^2, 2^2, -1]$
    \end{tabular}
    \medskip
    \caption{\label{table:primesquaremore}Some of the PWN with 2 and 3 square odd prime factors that we have found. Here, $w_6=2^2 \cdot 13 \cdot  17 \cdot 449 \cdot 24809 \cdot 223797481$, $w'_6 = 2^2 \cdot 13 \cdot  17 \cdot  443 \cdot 97919 \cdot 563915543$ and $\ap{i}$ denotes a prime with $i$ digits. Entries are in lexicographic order of index sequences.}
\end{table}

All of the above can be summed up in the following theorem:
\begin{thm}[PWN with non square-free odd part and $\Omega\le 7$]
		\label{thm:patterns}
		There are no PWN $m$ with a quadratic or higher power odd prime factor and $\Omega(m)<7$. There are no PWN $m$ with 2 quadratic odd prime factor and $\Omega(m)=7$. There are no PWN $m$ with a cubic or higher power odd prime factor and $\Omega(m)=7$.
\end{thm}

\section{Open problems}\label{sec:open}

	By examining Tables~\ref{table:primesquare} and~\ref{table:primesquaremore}, together with other weird numbers found by our search procedure and which may be found on-line, we observe some facts which can be useful for further experiments.
	
	First of all, there are some prefixes in the factorization which occur in many PWN. One of this recurring prefix is $2^2 \cdot 13 \cdot  17 \cdot  443 \cdot 97919$, which also leads to many PWN with 2 or more square odd prime factors.
	PWN with 2 square odd prime factors begin to appear in the results of the search procedure when $\Omega=12$, and become quite common when $\Omega=14$. It seems that increasing $\Omega$ makes the appearance of this kind of PWN easier. Since our search space is quite restricted, there are probably PWN with 2 square odd prime factors even for $\Omega < 12$, but we think they are quite rare.
	The same thing may be said about PWN with 3 square odd prime factors, which only appear with $\Omega=15$. Unfortunately, with $\Omega>15$ the numbers become huge (thousands of digits) and this makes experiments much more difficult.
	
	\begin{openq}
		For each $n\in\NN$, find a PWN with exactly $n$ square odd prime factors and the least $\Omega=\Omega_n$. From the previous section and Theorem~\ref{thm:patterns} we obtain $\Omega_1=7$, $8\le\Omega_2\le 12$, $8\le\Omega_3\le 15$, and in general if $n\ge 2$ we have $\Omega_n\ge 8$. 
	\end{openq}
	As mentioned, another question is the following.
	\begin{openq}
		Find a PWN with a cubic or higher power odd prime factor.
	\end{openq}
	
	From the experiments, odd square prime factors seems more common at the right end of the factorization, although in our search results they never appear in the last position.
	\begin{openq}
		Find a PWN which has its largest prime factor squared or to a higher power.
	\end{openq}
	
	On OEIS \href{https://oeis.org/A002975}{A002975} it was asked if the following fact is true: a weird number is primitive iff divided by its largest prime factor it is not weird. The following would be a counterexample.
	\begin{openq}
		Find a weird number $w$ which is not primitive and such that $w/(\text{largest prime factor})$ is not weird.
	\end{openq}

	The following problem appears as an editor's comment in~\cite{BenPSS}. Erd\H os offered 25\$ for its solution.	
	\begin{openq}
		Is $\sigma(m)/m$ bounded when $m$ ranges through the set of (not necessarily primitive) weird numbers?
	\end{openq}

	Finally, the following would settle a long-standing problem.
	\begin{openq}
	\item Find an odd weird number, or prove that all weird numbers are even.
	\end{openq}
	The above problem was raised by Erd\H os, that offered 10\$\ for an example of an odd weird number, and 25\$ for a proof that none can exist~\cite{BenPSS}. Wenjie Fang and Uwe Beckert proved, using parallel tree search, that there are no odd weird numbers up to $10^{21}$, and no odd weird numbers up to $10^{28}$ with abundance not exceeding $10^{14}$~\cite[Section 4.2]{FaBPTS}.

\emph{Acknowledgements.} We would like to thank Vincenzo Acciaro, Adam Atkinson, Rosa Gini, Francesca Scozzari, Agnese Telloni for their valuable discussions with us.
The second and third authors wish to thank the first and fourth authors for invitation and hospitality at Chieti-Pescara University, where significant parts of this work have been done.



\bibliographystyle{alpha}

\end{document}